\documentclass[11pt]{amsart}
\usepackage{graphicx}
\usepackage[all]{xy}
\usepackage[ansinew]{inputenc}
\usepackage{amsmath}
\usepackage{dsfont}
\usepackage{graphicx}
\usepackage{amsthm}
\usepackage{amsfonts}
\usepackage{newlfont}
\usepackage{amscd}
\usepackage{amssymb}
\usepackage{latexsym}
\usepackage{eufrak}
\usepackage{euscript}
\usepackage{mathrsfs}
\usepackage{color}
\usepackage{dsfont}
\usepackage[babel=true]{csquotes}
\usepackage[ansinew]{inputenc}
\usepackage{amsmath}
\usepackage{graphicx}
\usepackage{amsthm}
\usepackage{amsfonts}
\usepackage{newlfont}
\usepackage{amscd}
\usepackage{amssymb}
\usepackage{latexsym}
\usepackage{eufrak}
\usepackage{euscript}
\usepackage{mathrsfs}
\usepackage{color}
\usepackage{fancyhdr}
\usepackage{t1enc}
\usepackage{fancybox}
\usepackage{graphics}
\usepackage[toc,page]{appendix}
\usepackage{subfig}
\usepackage[all]{xy}
\usepackage{calligra}
\usepackage{enumitem}
\usepackage{cancel}
\usepackage{geometry}
\newtheorem{thm}{Theorem}[section]
\newtheorem{cor}[thm]{Corollary}
\newtheorem{lem}[thm]{Lemma}

\newtheorem{prop}[thm]{Proposition}
\theoremstyle{definition}
\newtheorem{defn}[thm]{Definition}
\newtheorem{defns}[thm]{Definitions}

\newtheorem{rem}[thm]{Remark}

\numberwithin{equation}{section}

\let \b = \beta

\let \r=\rho

\allowdisplaybreaks

\begin{document}
\title[Representations of Hom-Lie algebras] {representations of Simple Hom-Lie algebras}
\author[B. Agrebaoui, K. Benali and A. Makhlouf]{Boujemaa Agrebaoui$^1$, Karima Benali$^2$ and Abdenacer Makhlouf$^{3}$}%
\address{$^1$Université de Sfax, Faculté des Sciences de Sfax,
 Department de mathématiques, Route de la Soukra km 4 - Sfax - 3038, Tunisia}%
\email{b.agreba@fss.rnu.tn}
\address{$^2$Université de Sfax, Faculté des Sciences de Sfax,
 Department de mathématiques, Route de la Soukra km 4 - Sfax - 3038, Tunisia}%
\email{karimabenali172@yahoo.fr }
\address{Université de Haute Alsace
IRIMAS-d\'epartement de  Math\'ematiques, 
6, rue des Frères Lumière
68093 Mulhouse cedex, France}%
\email{abdenacer.makhlouf@uha.fr}

\thanks{}


\dedicatory{}
\begin{abstract}
The purpose  of this paper is to study representations of simple multiplicative Hom-Lie algebras. First, we provide a new  proof using Killing form  for  characterization theorem of simple Hom-Lie algebras given by Chen and Han,  then  discuss the representations structure of simple multiplicative Hom-Lie algebras. Moreover, we study weight modules and root space decompositions of simple multiplicative Hom-Lie algebras, characterize weight modules and provide examples of representations of  $\mathfrak{sl}_2$-type Hom-Lie algebras.
\end{abstract}

\maketitle \vspace{.20cm}
\vskip0.5cm

\date{\today}

\section{Introduction}
Nowadays, one of the most modern trends in mathematics has to do with
representations  and deformations. These two topics are
important tools in most parts of Mathematics and Physics. Hom-type algebras arised first in examples of $q$-deformations of algebras of vector fields, like Witt and Virasoro algebras, where the usual derivation is replaced by a $\sigma$-derivation. 
Motivated by these examples, Hartwig, Larsson and
Silvestrov developed from the algebraic point of view in \cite{HDS}  the deformation theory using $\sigma$-derivations and introduced a new category of algebras called Hom-Lie
algebras. A Hom-Lie algebra is a triple ($\mathfrak{g}$, $[\cdot ,\cdot ]_{\alpha}$, $\alpha$) in which the bracket satisfies a twisted Jacobi identity along the linear map $\alpha$. It should be pointed that
Lie algebras form a subclass of Hom-Lie algebras, i.e. when $\alpha$ equal to the identity map. Various classical algebraic structures where considered and generalized within similar framework  like Hom-Lie superalgebras in \cite{FA}.

 Representations  of Hom-Lie algebras were introduced and studied in \cite{YS}, see also \cite{SB&AM}. Based on this, we aim in this paper to discuss simple Hom-Lie algebras representations. Simple Hom-Lie algebras were characterized   in \cite{X.C&W.H}, where the authors showed that they are obtained by Yau twist of semisimple Lie algebras. This key observation is used here  to built a representation theory of simple Hom-Lie algebras.
Moreover examples are provided by a study of Hom-type $\mathfrak{sl}(2)$-modules. 

 The paper is organized as follows.  In Section
 2, we review  basic definitions and relevant constructions about Hom-Lie algebras. Then in Section 3 we recall some fundamental results about structure of simple multiplicative Hom-Lie algebras and provide a new proof of the main theorem using Killing form. 
Section 4 is dedicated to the construction of  multiplicative simple Hom-Lie
algebras representations, we  show that there is a correspondence between
representation of multiplicative simple Hom-Lie algebras and representation of the induced semisimple Lie
algebras using invertible
twisting maps.
In Section 5, we introduce and discuss the notion of simple multiplicative Hom-Lie algebra weight-modules.
Finally in Section 6 we study and classify Hom-$\mathfrak{sl}(2)$-modules.\\

\section{Basics}
In this section, we provide some preliminaries, basic definitions  and relevant constructions about Hom-Lie algebras and related structures. Throughout this paper all algebras and vector spaces are considered  over $\mathbb{K}$, an algebraically closed field of characteristic 0.

 \begin{defn}
A Hom-Lie algebra is a triple $(\mathfrak{g}, [\cdot ,\cdot]_{\alpha}, \alpha)$ consisting of
a vector space $\mathfrak{g}$, a bilinear map
$[\cdot ,\cdot ]_{\alpha}:\mathfrak{g}\times\mathfrak{g}\longrightarrow \mathfrak{g}$ and a linear
map $\alpha:\mathfrak{g}\rightarrow \mathfrak{g}$ that satisfies :
\begin{eqnarray*}
           & [x,y]_{\alpha}=-[y,x]_{\alpha},\;\forall x,y \in \mathfrak{g}\; \text{ (skewsymmetry)}\\
         &  \underset{x,y,z}\circlearrowleft[\alpha(x),[y,z]_{\alpha}]_{\alpha}=0, \forall x, y, z \in \mathfrak{g}\; \text{ (Hom-Jacobi identity)}.
            \end{eqnarray*}
A Hom-Lie algebra $(\mathfrak{g},[\cdot ,\cdot ]_{\alpha},\alpha)$ is said to be  \emph{multiplicative }
if $\alpha $
is an algebra morphism, i.e.
$$\alpha([x,y]_{\alpha})=[\alpha(x),\alpha(y)]_{\alpha},\forall x,y \in \mathfrak{g}.$$
It is said \emph{regular} if $\alpha$ is an algebra automorphism.
\end{defn}
\begin{defn}
Let $(\mathfrak{g},[\cdot ,\cdot ]_{\alpha},\alpha)$ be a Hom-Lie algebra. A subspace $\mathfrak{h}$ of $\mathfrak{g}$  is called
Hom-Lie \emph{subalgebra}  if $[\mathfrak{h},\mathfrak{h}]_{\alpha}\subseteq \mathfrak{h}$ and $\alpha(\mathfrak{h})\subseteq
\mathfrak{h}$. In particular, a Hom-Lie subalgebra $\mathfrak{h}$ is said to be an \emph{ideal} of
$\mathfrak{g}$ if $[\mathfrak{h},\mathfrak{g}]_{\alpha}\subseteq \mathfrak{h}$.
The Hom-Lie algebra $\mathfrak{g}$ is called  \emph{abelian}  if as usual $[x,y]=0, \forall x,y \in
\mathfrak{g}$.
\end{defn}
\begin{defn} Let $(\mathfrak{g}_{1},[\cdot ,\cdot ]_{\alpha_1},\alpha_{1})$ and $(\mathfrak{g}_{2},[\cdot ,\cdot ]_{\alpha_2},\alpha_{2})$
be two Hom-Lie algebras. A linear map $\varphi:\mathfrak{g}_{1}\longrightarrow \mathfrak{g}_{2}$ is a \emph{Hom-Lie algebra morphism}  if for all $x,y\in \mathfrak{g}_1$
$$\varphi([x,y]_{\alpha_1})=[\varphi(x),\varphi(y)]_{\alpha_2}  \text{ and } \varphi\circ \alpha_{1}=\alpha_{2}\circ \varphi .$$
In particular, they are isomorphic if $\varphi$ is a
bijective linear map.

A linear map $\varphi:\mathfrak{g}_{1}\longrightarrow \mathfrak{g}_{2}$ is said to be a \emph{weak Hom-Lie algebra morphism} if for all $x,y\in \mathfrak{g}_1$, we have only 
$\varphi([x,y]_{\alpha_1})=[\varphi(x),\varphi(y)]_{\alpha_2} $.
\end{defn}

There is a key construction introduced by D. Yau that gives rise to a Hom-Lie algebra starting from a Lie algebra and a Lie algebra homomorphism \cite{DY1}. 
\begin{prop}[Yau Twist]\label{twist}
 Let $\mathfrak{g}=(\mathfrak{g},[\cdot ,\cdot])$ be a Lie algebra and  $\alpha:\mathfrak{g} \longrightarrow \mathfrak{g}$ be a Lie algebra homomorphism. Then, $\mathfrak{g}_{\alpha}=(\mathfrak{g},[\cdot ,\cdot ]_{\alpha}:=\alpha([\cdot ,\cdot ])
 ,\alpha)$ is a Hom-Lie algebra.
\end{prop}
\begin{proof}
It is straightforward  to  prove that the new bracket  $[\cdot ,\cdot ]_{\alpha}$
satisfies the Hom-Jacobi identity.
\end{proof}
\begin{rem}
More generally, let  $(\mathfrak{g},[\cdot ,\cdot ]_{\alpha},\alpha)$ be a Hom-Lie algebra and
$\gamma:\mathfrak{g}\rightarrow \mathfrak{g}$ be a weak Hom-Lie algebra morphism. Then $\mathfrak{g}_{\gamma}=(\mathfrak{g},[\cdot ,\cdot ]_{\gamma}:=\gamma([\cdot ,\cdot ])
 ,\gamma\circ\alpha)$ is a new Hom-Lie algebra.
\end{rem}

\begin{defn}
Let $(\mathfrak{g},[\cdot ,\cdot ]_{\alpha},\alpha)$ be a Hom-Lie algebra.
If there exists a Lie algebra $(\mathfrak{g},[\cdot ,\cdot ])$ such that
$[x,y]_{\alpha}=\alpha([x,y])=[\alpha(x),\alpha(y)]$, for all $x,y \in \mathfrak{g}$, then
$(\mathfrak{g},[\cdot ,\cdot ]_{\alpha},\alpha)$ is said to be of Lie-type and $(\mathfrak{g},[\cdot ,\cdot ])$ is called
the induced Lie algebra of $(\mathfrak{g},[\cdot ,\cdot ]_{\alpha},\alpha)$.
\end{defn}
\begin{lem}
Let $(\mathfrak{g},[\cdot ,\cdot ]_{\alpha},\alpha)$ be a regular  Hom-Lie algebra. Then $(\mathfrak{g},[\cdot ,\cdot ]_{\alpha},\alpha)$ is of Lie-type with the induced Lie
algebra $(\mathfrak{g},[\cdot ,\cdot ])$, where $[x,y]=\alpha^{-1}([x,y]_{\alpha}), \forall x,y \in
\mathfrak{g}$.
\end{lem}
\begin{proof} Let $[x,y]=\alpha^{-1}([x,y]_{\alpha})$ for any $x,y\in \mathfrak{g}$.
Since $(\mathfrak{g},[\cdot ,\cdot]_{\alpha},\alpha)$ is multiplicative then,
$\alpha([x,y]_{\alpha})=[\alpha(x),\alpha(y)]_{\alpha}$.  It implies
$\alpha^{2}([x,y])=\alpha([\alpha(x),\alpha(y)])$ and thus
$\alpha([x,y])=[\alpha(x),\alpha(y)]$. In the following we shall prove that
$(\mathfrak{g},[\cdot ,\cdot ])$ is a Lie algebra. The skewsymmetry of  $[\cdot ,\cdot ]$ is obvious.
Now, we prove that it satisfies the Jacobi identity. Indeed,  let's $x,y,z \in \mathfrak{g}$, we have
\begin{eqnarray*}
\underset{x,y,z}\circlearrowleft[x,[y,z]]=\underset{x,y,z}\circlearrowleft
\alpha^{-1}[x,\alpha^{-1}[y,z]_{\alpha}]_{\alpha}
=\underset{x,y,z}\circlearrowleft \alpha^{-2}[\alpha(x),[y,z]_{\alpha}]_{\alpha}=0.
\end{eqnarray*}
It follows that $(\mathfrak{g},[\cdot ,\cdot ])$ is a Lie algebra.
\end{proof}

The concept of representation of a Hom-Lie algebra was introduced in \cite{YS}, see also \cite{SB&AM}.
\begin{defn}\label{sheng}
A representation of a multiplicative Hom-Lie algebra $(\mathfrak{g}, [\cdot ,\cdot ]_{\alpha}, \alpha)$
on the vector space V with respect to $\b \in End(V)$ is a linear map
$\rho_{\b}:\mathfrak{g}\longrightarrow End(V)$, such that for any $x,y \in \mathfrak{g},$ the following conditions are satisfied:
\begin{eqnarray}\label{ss}
  &  \rho_{\b}(\alpha(x))\circ\b=\b\circ\rho_{\b}(x),\\
\label{sss}&
    \rho_{\b}([x,y]_{\alpha})\circ\b=\rho_{\b}(\alpha(x))\rho_{\b}(y)
          -\rho_{\b}(\alpha(y))\rho_{\b}(x).
\end{eqnarray}
Hence $(V,\rho_{\beta},\beta)$ is called a $\mathfrak{g}$-module via the action
 $x.v=\rho_{\beta}(x)v, \forall x\in \mathfrak{g}, v\in V$.
\end{defn}
We have the following property.
\begin{prop}\label{alphabeta}
Let $(V,\rho_{\beta},\beta)$ be a representation of a simple multiplicative Hom-Lie algebra
$(\mathfrak{g},[\cdot,\cdot],\alpha)$ with $\b$  invertible. Then, $\forall n\in\mathbb{N}$ we have,
\begin{enumerate}
\item[(1)] $\rho_{\beta}(\alpha^{n}(x))=\beta^{n}\rho_{\beta}(x)\beta^{-n}.$
\item[(2)] $\rho_{\beta}(\alpha^{n}[x,y])\beta=\rho_{\beta}(\alpha^{n+1}(x))
          \rho_{\beta}(\alpha^{n}(y))
          -\rho_{\beta}(\alpha^{n+1}(y))\rho_{\beta}(\alpha^{n}(x)).$
\end{enumerate}
\end{prop}
\begin{proof}
$(1)$ is straightforward  by induction using (\ref{ss}) and 
$(2)$ is proved by induction using (\ref{sss}).
\end{proof}
\begin{defn}
For a $\mathfrak{g}$-module $(V,\r_{\b},\b)$, if a subspace $V_{1}\subseteq V$ is invariant
under $\rho_{\beta}(x), \forall x\in\mathfrak{g}$ and under $\b$ then $(V_{1},\r_{\b},\b\vert_{V_{1}})$ is called a $\mathfrak{g}$-submodule of $(V,\r_{\b},\b)$.\\
A $\mathfrak{g}$-module $(V,\r_{\b},\b)$ is called irreducible, if it has precisely two
$\mathfrak{g}$-submodules (itself and 0) and it is
called completely reducible if $V=V_{1}\oplus ...\oplus V_{s}$, where  $(V_{i},\r_{\b},\b\vert_{V_{i}})$ are irreducible $\mathfrak{g}$-submodules.
\end{defn}
\begin{thm}[\cite{XL}]\label{Keythm}
Let $(\mathfrak{g},[\cdot ,\cdot ]_{\alpha},\alpha)$ be a Lie-type Hom-Lie algebra  with
 $(\mathfrak{g},[\cdot ,\cdot ])$ the induced Lie algebra.
\begin{enumerate}
  \item Let $(V,\rho_{\b},\b)$ be  a representation
of the Hom-Lie algebra where $\b$ is invertible.
Then  $(V,\rho)=(V,\b^{-1}\circ\rho_{\b})$ is a representation of the Lie algebra $(\mathfrak{g},[\cdot ,\cdot ])$.
  \item Suppose that $(V,\rho)$ is a representation of the Lie algebra $(\mathfrak{g},[\cdot,\cdot])$.
  If there exists $ \b \in End(V)$ such that
  \begin{equation}\label{char}
\b\circ\rho(x)=\rho(\alpha(x))\circ\b, \forall x \in \mathfrak{g},\forall  v\in V,
      \end{equation}
  then $(V,\rho_{\b}=\b\circ\rho,\b)$ is a representation of the Hom-Lie algebra  $(\mathfrak{g},[\cdot ,\cdot ]_{\alpha},\alpha)$.
\end{enumerate}
\end{thm}
\begin{defn}
Let $(\mathfrak{g},[\cdot ,\cdot ]_{\alpha},\alpha)$ be a Lie-type Hom-Lie algebra  with
 $(\mathfrak{g},[\cdot ,\cdot ])$ the induced Lie algebra.
A representation $(V,\r_{\b},\b)$ of  $(\mathfrak{g},[\cdot ,\cdot ]_{\alpha},\alpha)$ is called of Lie-type
if $\r_{\b}=\b\circ\r$ where $\r$ is the representation of the induced Lie algebra $(\mathfrak{g},[\cdot,\cdot])$.
It is called regular if the representation
$(V,\r:=\b^{-1}\circ\r_{\b})$ of the induced Lie algebra $(\mathfrak{g},[\cdot,\cdot])$
is irreducible.
\end{defn}
The previous theorem provides a  relationship between representations
of Lie-type Hom-Lie algebras   and those of their induced Lie algebras.

\section{Structure of Simple Multiplicative Hom-Lie algebras}
In \cite{X.C&W.H}, the authors have proved that multiplicative simple Hom-Lie algebras are of Lie-type and their induced Lie algebras are semisimple. Moreover they discussed the dimension problem and showed that  there is an $n$-dimensional simple Hom-Lie algebra
for any integer $n$ larger than 2.  We should mention also the following  relevant references  \cite{SB&AM},\cite{XL} and \cite{DY3}.
\begin{defn}
A Hom-Lie algebra $(\mathfrak{g},[\cdot ,\cdot ]_{\alpha},\alpha)$  is
called  \emph{simple}  if it  has no proper ideals
and $[\mathfrak{g},\mathfrak{g}]_{\alpha}=\mathfrak{g}$. It  is called
 \emph{semisimple} Hom-Lie algebra if $\mathfrak{g}$ is a direct sum of certain simple ideals.
\end{defn}
We have the following  two propositions.
\begin{prop}[\cite{SB&AM}]
Let $(\mathfrak{g},[\cdot ,\cdot ])$ be a simple Lie algebra and let $\alpha \in Aut(\mathfrak{g})$. \\ Then,
$\mathfrak{g}_{\alpha}=(\mathfrak{g},\alpha([\cdot ,\cdot ]),\alpha)$ is a simple Hom-Lie algebra.
\end{prop}
\begin{prop}\label{morsimple}
Simple multiplicative Hom-Lie algebras are regular Hom-Lie algebras.
\end{prop}
\begin{proof}Let
$(\mathfrak{g},[\cdot ,\cdot ]_{\alpha},\alpha)$  be a simple  Hom-Lie algebra, then
$[\mathfrak{g},\mathfrak{g}]_{\alpha}=\mathfrak{g}$.
Suppose that  $ker(\alpha)\neq{0}$.  Then, $\alpha(ker(\alpha))=0$  and
$\alpha([ker(\alpha),\mathfrak{g}]_{\alpha})=0$.  So $ker(\alpha)$ is a non trivial ideal of
$(\mathfrak{g},[\cdot ,\cdot ]_{\alpha},\alpha)$. This contradicts  the simplicity of
$(\mathfrak{g},[\cdot ,\cdot ]_{\alpha},\alpha)$, except when $ker(\alpha)=0$.
Hence,  $\alpha$ is an automorphism.
\end{proof}
The following theorem summarizes results given in \cite{X.C&W.H} about  a characterization of simple multiplicative Hom-Lie algebras. We provide a new and different proof based on Killing form.
\begin{thm}(\cite{X.C&W.H})\label{simplethm}
Let $(\mathfrak{g},[\cdot ,\cdot ]_{\alpha},\alpha)$ be a simple multiplicative Hom-Lie algebra. Then the induced
Lie algebra $(\mathfrak{g},[\cdot ,\cdot ]=\alpha^{-1}([\cdot ,\cdot ]_{\alpha}))$ is semisimple
and its $n$ simple ideals are isomorphic mutually
 besides $\alpha$ acts simply transitively on simple ideals of $\mathfrak{g}$. Furthermore
$\mathfrak{g}$ can be generated by a simple ideal $\mathfrak{g}_{1}$ of the Lie algebra
$(\mathfrak{g},[\cdot ,\cdot ])$ and   $\alpha\in Aut(\mathfrak{g})$.
Taking $\alpha, \gamma \in Aut(\mathfrak{g})$ such that
$\alpha^{n}$ and $\gamma^{n}$ leaves each simple ideal invariant and
$\alpha^{n}(\mathfrak{g}_{1})=\mathfrak{g}_{1}$ (or
$\gamma^{n}(\mathfrak{g}_{1})=\mathfrak{g}_{1}$). Then we have
\begin{enumerate}
\item $\mathfrak{g}=\mathfrak{g}_{1}\oplus\alpha(\mathfrak{g}_{1})\oplus\alpha^{2}
(\mathfrak{g}_{1})\oplus...\oplus\alpha^{n-1}(\mathfrak{g}_{1})$
\Big(or $\mathfrak{g}=\mathfrak{g}_{1}\oplus\gamma(\mathfrak{g}_{1})\oplus\gamma^{2}
(\mathfrak{g}_{1})\oplus...\oplus\gamma^{n-1}(\mathfrak{g}_{1})$\Big).
\item $\alpha$ and $\gamma$ are conjugate on $\mathfrak{g}$ $\Leftrightarrow \alpha^{n}$ and $\gamma^{n}$ are conjugate on $\mathfrak{g}_{1}$.
    \end{enumerate}
\end{thm}
\begin{proof}
The Killing form $K:\mathfrak{g}\times \mathfrak{g}\rightarrow \mathbb{C}$ of $(\mathfrak{g},[\cdot,\cdot])$ is non-degenerate. In fact, let $\mathfrak{k}=\{x\in
\mathfrak{g}/K(x,y)=0,~\forall y\in \mathfrak{g}\} $ its kernel. It is clear that
$\mathfrak{k}$ is an ideal of $\mathfrak{g}$, since $K([x,y],z)=K(x,[y,z])=0,
~\forall x\in \mathfrak{k}, y,z\in \mathfrak{g}$. Since $\alpha$ is an automorphism and
$K(\alpha(x),\alpha(y))=K(x, y)$. Then $K(\alpha(x),y)=K(x, \alpha^{-1}(y))=0,~\forall
x\in \mathfrak{k}, y\in \mathfrak{g}$ and $\alpha(\mathfrak{k})\subset \mathfrak{k}$.
Then $\mathfrak{k}$ is an ideal of the multiplicative simple Hom-Lie algebra $\mathfrak{g}$
and then $\mathfrak{k}= 0$ and $K$ is non-degenerate. We deduce that $(\mathfrak{g},[\cdot,\cdot])$
is a semisimple Lie algebra and then a direct sum of its simple ideals.
Let $\mathfrak{g}_{1}$ be a minimal proper ideal of the induced Lie algebra $(\mathfrak{g},[\cdot,\cdot])$. In particular $\mathfrak{g}_{1}$ is a simple ideal. Let $n$ be the minimal integer such that $\alpha^{n-1}(\mathfrak{g}_1)\neq \mathfrak{g}_1$ and $\alpha^n(\mathfrak{g}_{1})= \mathfrak{g}_{1}$. The algebra $\mathfrak{b}=\mathfrak{g}_{1}\oplus\alpha(\mathfrak{g}_{1})\oplus\ldots\oplus
\alpha^{n-1}(\mathfrak{g}_{1})$ is an ideal of the simple multiplicative Hom-Lie algebra $(\mathfrak{g},[.,.]_\alpha,\alpha)$. Then $\mathfrak{g}=\mathfrak{b}=\mathfrak{g}_{1}\oplus\alpha(\mathfrak{g}_{1})\oplus\ldots
\oplus\alpha^{n-1}(\mathfrak{g}_{1}).$
\end{proof}
\begin{defn}
A multiplicative Hom-Lie algebra is called  semisimple if its induced Lie algebra is semisimple.
\end{defn}
\begin{prop}
A multiplicative semisimple Hom-Lie algebra is a direct sum of multiplicative simple Hom-Lie algebras.
\end{prop}
\begin{proof}
For $\mathfrak{g}$ semisimple Hom-Lie algebra, the Killing form of the induced Lie algebra will be non-degenerate and from the previous definition the induced Lie algebra
 is semisimple. We take a minimal ideal  $\mathfrak{g}_{1}$ which will be a simple ideal. Let $n$ be the minimal integer such that $\alpha^{n-1}(\mathfrak{g}_{1})\neq \mathfrak{g}_{1}$ and $\alpha^n(\mathfrak{g}_{1})= \mathfrak{g}_{1}$. The algebra $\mathfrak{b}=\mathfrak{g}_{1}\oplus\alpha(\mathfrak{g}_{1})\oplus\ldots\oplus\alpha^{n-1}
 (\mathfrak{g}_{1})$ is a simple ideal of the semisimple multiplicative Lie algebra. Let $\mathfrak{b}'$ be the subspace of $\mathfrak{g}$ orthogonal (with respect to  $K$) to $\mathfrak{b}$. As $K$ is invariant, $\mathfrak{b}'$ is an ideal of $\mathfrak{g}$. In fact, let $x \in  \mathfrak{b}, y \in  \mathfrak{b}'$, $z\in \mathfrak{g}$, by invariance of $K$, we have $K(x, [z, y])=K([x, z], y)= 0$ since $[x, z]\in \mathfrak{b}$. Moreover, since $\alpha(\mathfrak{b})=\mathfrak{b}$, $K(x,\alpha(y))=K(\alpha^{-1}(x),y)=0$, then $\mathfrak{b}'$ is an ideal of the Hom-Lie algebra $(\mathfrak{g},[.,.]_\alpha,\alpha)$.

 By minimality of $\mathfrak{b}$, the intersection $\mathfrak{b}\cap \mathfrak{b}'$ can only be $(0)$ or $\mathfrak{b}$. We can prove that  the second case cannot occur. If not $K(x,y)=0, \forall x,y\in \mathfrak{b}$ and $x=\sum_{i=1}^k[x_i,y_i]$ since $[\mathfrak{b},\mathfrak{b}]=\mathfrak{b}$.Then for all $z\in \mathfrak{g}$, using the invariance of $K$ and $\mathfrak{b}$ an ideal, we have
 $$K(x,z)=K(\sum_{i=1}^k[x_i,y_i],z)=\sum_{i=1}^{k}K([x_i,y_i],z)=\sum_{i=1}^{k}
 K(x_i,[y_i,z])=0$$ which contradicts the fact that $K$ is non-degenerate. Hence $\mathfrak{g}=\mathfrak{b}\oplus\mathfrak{b}'$.
The restriction of $K$ to $\mathfrak{b}'\times\mathfrak{b}'$ is invariant non-degenerate bilinear form.\\ The proof is completed by induction on the dimension of $\mathfrak{g}$.
\end{proof}

\section{Representations of simple multiplicative Hom-Lie algebras}
We aim in this section to characterize representations of simple multiplicative Hom-Lie algebras and provide the relationship with those  of the
induced semisimple Lie algebra.
\begin{prop}
Let $(V,\r_{\b},\b)$ be a representation of a simple multiplicative Hom-Lie algebra $(\mathfrak{g},[\cdot,\cdot]_{\alpha},\alpha)$. Then, $Im(\b)$ and $Ker(\b)$ are
submodules of $V$ for $(\mathfrak{g},[\cdot,\cdot]_{\alpha},\alpha)$. Moreover,
we have an isomorphism of $(\mathfrak{g},[\cdot,\cdot]_{\alpha},\alpha)$-modules
$\overline{\b}:V/Ker(\b)\longrightarrow Im(\b).$
\end{prop}
\begin{proof}
Let $v\in Ker(\b), \r_{\b}(\alpha(x))\circ\b(v)=0=\b(\r_{\b}(x)(v)),
\forall x\in\mathfrak{g}.$ Then, $\r_{\b}(\alpha(x))(v)\in Ker(\b).$ So $Ker(\b)$ is a submodule of $V$. 
Now let $v\in Im(\b)$, there exists $w\in V$ such that $v=\b(w).$ 
Since $\alpha$ is an automorphism $\forall x\in \mathfrak{g},
\rho_{\b}(x)(v)=\r_{\b}(\alpha(\alpha^{-1}(x)))\b(w)=\b(\r_{\b}(\alpha^{-1}(x)w)$.
So $\r_{\b}(x)(v)\in Im(\b), \forall x\in \mathfrak{g}$ and therefore
$Im(\b)$ is a submodule of V.
\end{proof}
\begin{cor}
If $(V,\r_{\b},\b)$ is an irreducible representation of a simple multiplicative
Hom-Lie algebra $(\mathfrak{g},[\cdot,\cdot]_{\alpha},\alpha)$. Then
$\b$ is invertible.
\end{cor}

\begin{prop}
Let $(\mathfrak{g},[\cdot,\cdot]_{\alpha},\alpha)$ be a simple multiplicative Hom-Lie
algebra and $(V,\r_{\b},\b)$ a representation with $\b$ invertible.
If $(V,\r:=\b^{-1}\circ\r_{\b})$ is irreducible representation  of  the induced Lie algebra. Then
$(V,\r_{\b},\b)$ is irreducible representation of  the multiplicative simple Hom-Lie algebra.
\end{prop}
\begin{proof}
Assume that $(V,\r_{\b},\b)$ is reducible. Then, there exists $W\neq\{0_{V}\}$ a subspace
of $V$ such that $(V,\r_{\b},\b\mid_{W})$ is a submodule of
$(\mathfrak{g},[\cdot,\cdot]_{\alpha},\alpha)$.
That is $\b(W)\subset W$ and $\r_{\b}(x)W\subset W, \forall x\in \mathfrak{g}$.
Hence, $\b\circ\r(x)W\subset W, \forall x\in \mathfrak{g}$ and then
$\r(x)W\subset \b^{-1}(W)\subset W, \forall x\in \mathfrak{g}$ and so $W$
is a submodule for $(V,\rho)$ which is a contradiction.
\end{proof}
\begin{prop}\label{prop4.4}
Let $\mathfrak{g}$ be a simple Lie algebra and $(W,\r)$ be a representation of $\mathfrak{g}$.
Then, for $\alpha \in Aut(\mathfrak{g}),~~(W,\widetilde{\r}:=\r\circ\alpha^{-k})$
is a representation of $\alpha^{k}(\mathfrak{g}).$ Moreover, there exists $\b\in GL(W)$
such that $\widetilde{\r}(\alpha^{k}(x))=\b^{k}\circ \r_{k}(x)\circ\b^{-k},
\forall x\in \mathfrak{g}$.
\end{prop}
\begin{proof}
Let $(W,\r)$ be a representation of $\mathfrak{g}$ on $W$ such that the following diagram
is commutative
\begin{equation}
 \begin{array}{clcrl}
\mathfrak{g} &  \overset{\r}{\longrightarrow}   & End(W) \\
      \alpha^{k}\Big\downarrow    &\,&  \Big\downarrow T \\
    \alpha^{k}(\mathfrak{g}) & \overset{ \widetilde{\r}}\longrightarrow  & End(W)
\end{array}
\end{equation}
That is there exists $T$ such that $\widetilde{\r}\circ\alpha^{k}=T\circ\r$. Using Skolem-Noether Theorem \cite{B}, there exists $S\in GL(W)$ such that
$\forall x\in \mathfrak{g},~~ T\circ\r(x)=S\circ\r(x)\circ S^{-1}$. Then, we have
$\widetilde{\r}\circ \alpha^{k}(x)=S\circ\r(x)\circ S^{-1}$.\\
By basic linear algebra theory, there exists $\b\in GL(W)$ such that $S=\b^{k}$. Then the commutativity becomes
$\widetilde{\r}\circ\alpha^{k}(x)=\b^{k}\circ\r(x)\circ\b^{-k}, \forall x\in\mathfrak{g}.$
\end{proof}

 \begin{prop}(\cite{B1})\label{prop4.5}
 Let $\mathfrak{g}=\mathfrak{g}_{1}\oplus...\oplus
\mathfrak{g}_{n}$ be a  semisimple Lie algebra and let $(V_{i},\r_{i})$ be a $\mathfrak{g}_{i}$-module, $\forall~~ 0\leq i\leq n$. Then, an irreducible representation of $\mathfrak{g}$ is given by $(V,\rho)$ where $V=V_{1}\otimes...\otimes V_{n}$
 and $\rho$ is given for $x=(x_1,...,x_n)$ by
\begin{eqnarray*}
       \r:\mathfrak{g} &\longrightarrow&End(V_{1}\otimes...\otimes V_{n})  \\
       x &\longmapsto& \r(x)=\displaystyle\sum^{n}_{i=1}
       Id_{V_1}\otimes...\otimes\r_{i}(x_{i})\otimes...\otimes Id_{V_{n}}.
     \end{eqnarray*}
     \end{prop}
Using the classification of simple multiplicative Hom-Lie algebras
 by Chen and Han in Theorem \ref{simplethm},  Li Theorem  \ref{Keythm} and Proposition
 \ref{prop4.4}, we provide a  construction of representations of simple multiplicative
 Hom-Lie algebras.
\begin{thm}\label{thmbase}
Let $(\mathfrak{g}, [\cdot ,\cdot ]_{\alpha}, \alpha)$ be a simple multiplicative Hom-Lie algebra
and $\mathfrak{g}_{1}$ be a simple ideal generating the induced semisimple Lie algebra such that
$\mathfrak{g}=\mathfrak{g}_{1}\oplus\alpha(\mathfrak{g}_{1})\oplus...\oplus \alpha^{n-1}(\mathfrak{g}_{1})$.\\ Let $(V_0,\rho_0),~(V_1,\rho_1),\ldots,(V_{n-1},
\rho_{n-1})$ be $n$ representations of $\mathfrak{g}_{1}$.
Then,
\begin{itemize}
\item[(1)] There exist $\beta_0\in GL(V_0),
\beta_1\in GL(V_1),\ldots, \beta_{n-1}\in GL(V_{n-1})$ depending on $\alpha$ such that \\ $\rho:\mathfrak{g}\rightarrow End(V)$, where
$V=V_{0}\otimes...\otimes V_{n-1}$, defined for all $X=(x_0, \alpha(x_1),
\ldots,\alpha^{n-1}(x_{n-1}))\in \mathfrak{g}$ by
\begin{equation}\label{simplerep}
\rho(X)=\displaystyle\sum_{k=0}^{n-1} Id_{V_0}\otimes\ldots\otimes\Big(\beta_k^k\circ\rho_{k}(x_{k})\circ\beta_k^{-k}\Big)
\otimes\ldots\otimes Id_{V_{n-1}},
\end{equation}
 is an irreducible representation of the induced Lie algebra $(\mathfrak{g},[\cdot,\cdot])$ on $V$.
\item[(2)]Let $\beta:=\beta_{0}\otimes...\otimes\beta_{n-1}\in GL(V)$ and
let $\r_{\b}:\mathfrak{g}\longrightarrow End(V)$ defined by
\begin{equation}\label{Homsimplerep}
\rho_{\beta}(X)=\b\circ\r(X)= \displaystyle\sum^{n-1}_{k=0}
       \b_{0}\otimes...\otimes\b^{k+1}_{k}\circ\rho_{k}(x_{k})\circ\b^{-k}_{k}
\otimes...\otimes \b_{n-1}.
\end{equation}
Then the triple $(V,\rho_{\beta},\beta)$
is a regular representation of the simple multiplicative Hom-Lie algebra $(\mathfrak{g}, [\cdot ,\cdot ]_{\alpha}, \alpha)$.
\end{itemize}
\end{thm}

\begin{proof}
Let $\r:\mathfrak{g}_{1}\longrightarrow End(W)$ be a representation of $\mathfrak{g}_{1}$.
By Proposition \ref{prop4.4}, for all $k\in \mathbb{N}$, there exists $\b_{k}\in GL(W)$
and a representation $\widetilde{\r}_{k}$ of $\alpha^{k}(\mathfrak{g}_{1})$ on $W$ given by
$\widetilde{\r}_{k}(\alpha^{k}(x))=\b_{k}^{k}\circ\r_{k}(x)\circ\b^{-k}_{k}$.\\
Now for $(V_0,\rho_0),~(V_1,\rho_1),\ldots,(V_{n-1},
\rho_{n-1})$  $n$ representations of $\mathfrak{g}_{1}$, where $\mathfrak{g}=\mathfrak{g}_{1}
\oplus\alpha(\mathfrak{g}_{1})\oplus...\oplus\alpha^{n-1}(\mathfrak{g}_{1})$ and
$V=V_{0}\otimes...\otimes V_{n-1}$. Then, according to Proposition \ref{prop4.4} and Proposition \ref{prop4.5}, there exists $\b_{0}\in GL(V_{0}),...,\b_{n-1}\in GL(V_{n-1})$ such that $\r:\mathfrak{g}\longrightarrow End(V)$
defined by \[\begin{array}{lll}
\rho(x_0, \alpha(x_1),\ldots,\alpha^{n-1}(x_{n-1}))\\=\rho_0(x_0)\otimes Id_{V_1}\otimes\ldots\otimes Id_{V_{n-1}}+
Id_{V_0}\otimes\Big(\beta_1\circ\rho_1(x_1)\circ\beta_1^{-1}\Big)\otimes\ldots\otimes Id_{V_{n-1}}+\\
\ldots+Id_{V_0}\otimes Id_{V_1}\otimes\ldots\otimes \Big(\beta_{n-1}^{n-1}\circ\rho_{n-1}(x_{n-1})\circ\beta_{n-1}^{-n+1}\Big)\\
=\displaystyle\sum_{k=0}^{n-1} Id_{V_0}\otimes\ldots\otimes\Big(\beta_k^k\circ\rho_{k}(x_{k})\circ\beta_k^{-k}\Big)
\otimes\ldots\otimes Id_{V_{n-1}}
\end{array}\] is an irreducible representation of the induced Lie algebra $(\mathfrak{g},[\cdot,\cdot])$ on $V$,
which completes the proof of $(1).$\\
For $(2)$,  set $\beta:=\beta_{0}\otimes...\otimes \beta_{n-1}$. Condition \eqref{char}
is satisfied, that is for $(v_{0}\otimes...\otimes v_{n-1})\in V$, we have
\begin{small}
\begin{eqnarray*}
\b\circ \r(X)(v_{0}\otimes...\otimes
v_{n-1})&=&\b\circ\displaystyle\sum^{n-1}_{k=0}Id(v_{0})\otimes...\otimes
\b_{k}^{k}\circ\r_{k}(x_{k})\circ\b_{k}^{-k}(v_{k})
\otimes...\otimes Id(v_{n-1})\\
&=&\displaystyle\sum^{n-1}_{k=0}\b_{0}(v_{0})\otimes...\otimes \b_{k}^{k+1}\circ\r_{k}(x_{k})\circ\b_{k}^{-k}(v_{k})
\otimes...\otimes \b_{n-1}(v_{n-1})\\
&=&\displaystyle\sum^{n-1}_{k=0}\b_{0}(v_{0})\otimes...\otimes \b_{k}^{k+1}\circ\r_{k}(x_{k})\circ\b_{k}^{-k}\circ\b_{k}\circ\b_{k}^{-1}(v_{k})
\otimes...\otimes \b_{n-1}(v_{n-1})\\
&=&\displaystyle\sum^{n-1}_{k=0}\b_{0}(v_{0})\otimes...\otimes \b_{k}^{k+1}\circ\r_{k}(x_{k})\circ\b_{k}^{-(k+1)}\circ\b_{k}(v_{k})\otimes...\otimes \b_{n-1}(v_{n-1})\\
&=&\displaystyle\sum^{n-1}_{k=0}\b_{0}(v_{0})\otimes...\otimes
\r_{k}(\alpha^{k+1}(x_{k}))\circ\b_{k}
\otimes...\otimes \b_{n-1}(v_{n-1})\\
&=&\r(\alpha(X))\circ\b(v_{0}\otimes...\otimes v_{n-1}).
\end{eqnarray*}
\end{small}
So $(V,\rho_{\beta}=\beta\circ\rho,\b)$ is a representation of the simple multiplicative Hom-Lie
algebra $(\mathfrak{g},[\cdot,\cdot]_{\alpha},\alpha)$, which is regular if the   representation of the induced Lie algebra is irreducible.
\end{proof}

\begin{prop}
Let $(\mathfrak{g},[\cdot,\cdot]_{\alpha},\alpha)$ and $(\mathfrak{g},[\cdot,\cdot]_{\gamma},\gamma)$ be two isomorphic simple multiplicative
Hom-Lie algebras generated by a simple ideal $\mathfrak{g}_{1}$. Let $(\mathfrak{g},[\cdot,\cdot]_{\alpha}^{'})$ and $(\mathfrak{g},[\cdot,\cdot]_{\gamma}^{'})$
their induced Lie algebras.\\
If $(V,\r_{\b},\b)$ (resp. $(V,\r_{\delta},\delta)$) is an irreducible representation of $(\mathfrak{g},[\cdot,\cdot]_{\alpha},\alpha)$ (resp.
 $(\mathfrak{g},[\cdot,\cdot]_{\gamma},\gamma)$) on the same vector space
 $V$. Then $\delta$ and $\b$ are conjugated
 by $S\in GL(V)$.
\end{prop}

\begin{proof}
Since $(\mathfrak{g},[\cdot,\cdot]_{\alpha},\alpha)$ and
$(\mathfrak{g},[\cdot,\cdot]_{\gamma},\gamma)$ are isomorphic, then there exists a linear automorphism $\varphi$ of $\mathfrak{g}$
 such that $\varphi \circ \alpha=\gamma\circ\varphi$ and $\varphi([x,y]_\alpha)=[\varphi(x),\varphi(y)]_\gamma$. On the induced Lie algebras $(\mathfrak{g},[\cdot,\cdot]_{\alpha}^{'})$ and $(\mathfrak{g},[\cdot,\cdot]_{\gamma}^{'})$, $\varphi$ is still an isomorphism of semisimple Lie algebras that is, it satisfies $\varphi([x,y]'_\alpha)=[\varphi(x),\varphi(y)]'_\gamma$.

Let $(V,\r_{\b},\b)$ and $(V,\r_{\delta},\delta)$ be two representations on the vector space $V$ of respectively $(\mathfrak{g},[\cdot,\cdot]_{\alpha},\alpha)$ and $(\mathfrak{g},
[\cdot,\cdot]_{\gamma},\gamma)$.
The corresponding induced representations $(V,\r_{\alpha}^{'})$
and $(V,\r_{\gamma}^{'})$  respectively of $(\mathfrak{g},[\cdot,\cdot]_{\alpha}^{'})$ and $(\mathfrak{g},[\cdot,\cdot]_{\gamma}^{'})$ are equivalent. So there exists  a linear automorphism $T$ of $End(V)$ such that $\r_{\gamma}^{'}\circ \varphi
=T\circ\r_{\alpha}^{'}$. By Skolem-Noether Theorem \cite{B}, there exists an $S\in GL(V)$ such that $T=Ad_{S}$. Therefore,
\begin{equation}\label{skolamn}
\forall x\in \mathfrak{g}, \r_{\gamma}^{'}(\varphi(x))=
S\circ\r_{\alpha}^{'}(x)\circ S^{-1}.
\end{equation}
Using Theorem \ref{Keythm}, Theorem \ref{simplethm} and Proposition \ref{alphabeta},
\begin{eqnarray*}
\delta^{2}\circ\r_{\gamma}^{'}(\varphi(x))\circ\delta^{-1}&=&\r_{\delta}(\gamma(\varphi(x)))
=\r_{\delta}(\varphi(\alpha(x)))
=\delta\circ\r_{\gamma}^{'}(\varphi(\alpha(x)))
=\delta\circ S\circ\r_{\alpha}^{'}(\alpha(x))\circ S^{-1}\\
&=&\delta\circ S\circ \beta \circ \r^{'}_{\alpha}(x)\circ \b^{-1}\circ S^{-1}.
\end{eqnarray*}
So we get
\begin{equation}\label{egalite1}
\delta^{2}\circ\r_{\gamma}^{'}(\varphi(x))=
\delta\circ S\circ \beta \circ \r^{'}_{\alpha}(x)\circ \b^{-1}\circ S^{-1}.
\end{equation}
Moreover, using (\ref{skolamn}) we get
\begin{equation}\label{egalite2}
 \delta^{2}\circ\r_{\gamma}^{'}(\varphi(x))=\delta^{2}\circ S\circ\r_{\alpha}^{'}(x)\circ S^{-1}
 \circ\delta^{-1}.
\end{equation}
Comparing (\ref{egalite1}) and (\ref{egalite2}), we will obtain $S\circ\b=\delta\circ S$.
Therefore $\delta=S\circ\b\circ S^{-1}$.
\end{proof}

\section{Weight modules of  Simple multiplicative Hom-Lie algebras}We introduce and discuss in the following the root space decomposition of simple multiplicative Hom-Lie algebras,  weight modules and Verma modules.
\subsection{Root space decomposition of simple multiplicative Hom-Lie algebras}
Let $(\mathfrak{g},[\cdot,\cdot]_{\alpha},\alpha)$ be a simple multiplicative Hom-Lie algebra and let $\mathfrak{g}_{1}$ be a simple ideal generating the induced Lie algebra $(\mathfrak{g},[\cdot,\cdot])$ such that $\mathfrak{g}=\mathfrak{g_1}\oplus \alpha(\mathfrak{g_1})\ldots \oplus\alpha^{n-1}(\mathfrak{g_1})$.\\
Let $\mathfrak{h}_{1}$
be a Cartan subalgebra of $\mathfrak{g}_{1}$ (we write CSA for abbreviation).
The decomposition
of $\mathfrak{g}_{1}$ into root spaces relatively to $\mathfrak{h}_{1}$ is given by
 $\mathfrak{g}_{1}=\mathfrak{h}_{1}\oplus
\underset{\eta\in\Delta_{1}}\bigoplus (\mathfrak{g}_{1})_{\eta},$ where $\Delta_{1}$ is the set of roots.
Since $\alpha$ is an automorphism of $\mathfrak{g}$. For all
$\eta\in(\mathfrak{g}_{1})_{\eta}, x\in\mathfrak{g}_{1}$ and $h\in \mathfrak{h}_{1}$,
we have
 $\alpha^{k}([h,x])=\eta(h)\alpha^{k}(x),\\ \forall 0\leq k\leq n-1$. Namely,
$[\alpha^{k}(h),\alpha^{k}(x)]=\eta(h)\alpha^{k}(x)$ and so
$[\alpha^{k}(h),\alpha^{k}(x)]=\eta\circ\alpha^{-k}(\alpha^{k}(h))\alpha^{k}(x)$. Then,
$\alpha^{k}((\mathfrak{g}_{1})_{\eta})$ is a root space of $\alpha^{k}(\mathfrak{g}_{1})$
and we have
$\alpha^{k}((\mathfrak{g}_{1})_{\eta})=(\alpha^{k}(\mathfrak{g}_{1}))_{\eta\circ
\alpha^{-k}}$. The set of roots of $\alpha^{k}(\mathfrak{g}_{1})$
is given by $\Delta_{k+1}=\Delta_{1}\circ\alpha^{-k}
=\{\eta\circ\alpha^{-k}, \eta\in \Delta_{1}\}$.
\begin{rem}
If $\alpha^{n}$ is an outer automorphism, then $\alpha^{n}(\mathfrak{h}_{1})=\mathfrak{h}_{1}$, otherwise, $\alpha^{n}(\mathfrak{h}_{1})\neq\mathfrak{h}_{1}$.
\end{rem}
A consequence of the discussion above is the following proposition.
\begin{prop}
Let $\mathfrak{g}=\mathfrak{g}_1\oplus \alpha(\mathfrak{g_1})\ldots \oplus\alpha^{n-1}(\mathfrak{g_1})$ be the induced Lie algebra of a simple
multiplicative Hom-Lie algebra $(\mathfrak{g},[\cdot,\cdot]_{\alpha},\alpha).$ Then\begin{itemize}
\item[(1)]
$\mathfrak{h}=\mathfrak{h}_{1}\oplus \alpha(\mathfrak{h}_{1})\oplus...\oplus
 \alpha^{n-1}(\mathfrak{h}_{1})$ is a CSA of $\mathfrak{g}$ and $\Delta=\displaystyle
 \bigsqcup_{k=0}^{n-1} \Delta_{k+1}$ is the set of roots
  with respect to $\mathfrak{h}$ and the root space decomposition is given by
  $\mathfrak{g}=\mathfrak{h}\oplus\displaystyle\bigoplus^{n-1}_{k=0}\underset{\eta\in
  \Delta_{k+1}}\bigoplus\alpha^{k}(\mathfrak{g}_{1})_{\eta}.$
  \item[(2)]Let $\gamma$ be another automorphism such that $\mathfrak{g}=\mathfrak{g_1}\oplus \gamma(\mathfrak{g_1})\ldots \oplus\gamma^{n-1}(\mathfrak{g_1}), \gamma^{n}(\mathfrak{g}_{1})=\mathfrak{g}_{1}$. Then, there  exists
$\varphi\in Aut(\mathfrak{g})$ such that $\varphi\circ\alpha=\gamma\circ\varphi,
 (\gamma=\varphi\circ\alpha\circ\varphi^{-1})$ and \\ 
 $\gamma^{k}(\mathfrak{g}_{1})=\gamma^{k}(\mathfrak{h}_{1})\oplus \underset{\eta\in\Delta_{1}\circ\gamma^{-k}}\bigoplus(\gamma^{k}(\mathfrak{g}_{1}))_{\eta}$. \\ Set
$\Delta^{'}_{k+1}=\Delta_{1}\circ\gamma^{-k}=\Delta_{1}\circ\varphi\circ\alpha^{-k}
\circ\varphi^{-1}$ and $\mathfrak{h}^{'}=\mathfrak{h}_{1}\oplus\gamma(\mathfrak{h}_{1})
\oplus...\oplus\gamma^{n-1}(\mathfrak{h}_{1})$.\\  Hence $\mathfrak{g}=\mathfrak{h}\oplus(\displaystyle \bigoplus_{k=0}^{n-1}
\underset{\eta\in\Delta_{k+1}^{'}}\bigoplus(\gamma^{k}(\mathfrak{g}_{1}))_{\eta})$
is the root space decomposition of $\mathfrak{g}$ with respect to $\mathfrak{h}^{'}.$
  \end{itemize}
\end{prop}
\subsection{Weight modules of simple multiplicative Hom-Lie algebras}
Let $\mathfrak{g}=\mathfrak{g}_{1}\oplus...\oplus\mathfrak{g}_{n}$ be a semisimple Lie
algebra.  Let $\mathfrak{h}=\mathfrak{h}_{1}\oplus\mathfrak{h}_{2}\oplus...
\oplus\mathfrak{h}_{n}$ be a CSA of $\mathfrak{g}$ and
$\mathfrak{g}=\mathfrak{n}^{+}\oplus\mathfrak{h}\oplus\mathfrak{n}^{-}$ its triangular
decomposition. Let $\Delta$ be the set of roots with respect to $\mathfrak{h}$. Let $\rho:\mathfrak{g}\longrightarrow
End(V)$ be a representation of $\mathfrak{g}$.
Let $\lambda\in\mathfrak{h}^{*}$ and $V_{\lambda}=\{v\in V: \r(h)v=\lambda(h) v,\forall h\in \mathfrak{h}\}$. If $V_{\lambda}\neq \{0\}, \lambda$ is called a weight of $V$
and $V_{\lambda}$ is called a weight subspace of $V$ of weight $\lambda.$
 Denote the set of all weights by $P(V)$.
If $V=\underset{\lambda\in P(V)}\bigoplus V_{\lambda}$ is a direct sum of its weight subspaces then, we say that $V$ is a weight module.
Every weight module of $\mathfrak{g}$ is a sum of modules of the form
$V_{1}\otimes...\otimes V_{n}$ where $V_{i}$ is a module of the simple factor $\mathfrak{g}_{i}$ of $\mathfrak{g}$. Let $V=V_{1}\otimes...\otimes V_{n}$, then the set of weights of $V$ is given by:
$$P(V)=P(V_{0})\oplus...\oplus P(V_{n-1})=\{\displaystyle \sum_{k=1}^{n}\lambda_{i}/
 \lambda_{i}\in P(V_{i})\}$$ and we have
$$V=\underset{(\lambda_{1},...,\lambda_{n})\in P(V)}\oplus V_{\lambda_{1}}\otimes...\otimes V_{\lambda_{n}}.$$
Let $(W,\r)$ be a representation  of $\mathfrak{g}_{1}$. By Proposition \ref{prop4.4},  we associate a representation
$(W,\widetilde{\r})$ of $\alpha^{k}(\mathfrak{g}_{1})$ such that there exists $\b\in GL(W)$
satisfying $\widetilde{\r}(\alpha^{k}(x))=\b^{k}\circ\r_{k}\circ\b^{-k}$.
\begin{prop}\label{weightmodul}
Let $(\mathfrak{g},[\cdot,\cdot]_{\alpha},\alpha)$ be a simple multiplicative Hom-Lie algebra
and  $\mathfrak{g}=\mathfrak{g}_{1}\oplus\alpha(\mathfrak{g}_{1})\oplus\alpha^{2}
(\mathfrak{g}_{1})\oplus...\oplus\alpha^{n-1}(\mathfrak{g}_{1})$ be its induced Lie algebra,
$(\alpha^{n}(\mathfrak{g}_{1})=\mathfrak{g}_{1}).$
\begin{itemize}
\item[(1)]Consider $n$ weight modules $(V_{0},\r_{0}), (V_{1},\r_{1}),...,
(V_{n-1},\r_{n-1})$ of $\mathfrak{g}_{1}$. Then, for all integer   $k=0,\cdots, n-1$, there exists a
representation $(V_{k},\widetilde{\r}_{k})$ of $\alpha^{k}(\mathfrak{g}_{1})$ on $V_{k}$
and $\beta_{k} \in GL(V_{k})$ satisfying $\widetilde{\r}_{k}(\alpha^{k}(x))=\b_{k}^{k}\circ\r_{k}(x)\circ\b_{k}^{-k}$  for all $x\in\mathfrak{g}_{1}.$ Then
$V_{k}$ is a weight module of $\alpha^{k}(\mathfrak{g}_{1})$ and $\b_{k}^{k}((V_{k})_{\lambda})$ is a weight subspace of $V_{k}$ of weight $\lambda\circ\alpha^{-k}.$
\item[(2)]Let $P(V_{i})$ be the set of weights of $V_{i}$ with respect to $\mathfrak{h}_{1}$. The set of weights of $V=\displaystyle \otimes^{n-1}_{i=0}V_{i}$ with respect to $\mathfrak{h}=\mathfrak{h}_{1}\oplus
    \alpha(\mathfrak{h}_{1})\oplus...\oplus\alpha^{n-1}(\mathfrak{h}_{1})$ is given by
 $$P_{\mathfrak{h}}(V)=\Big\{\displaystyle \sum_{k=0}^{n-1}\lambda_{i}\circ\alpha^{-k}/ \lambda_{i}\in P(V_{i})\Big\}.$$
 The weight decomposition of $V$ is given by
 $$V=\underset{\lambda\in P(V)}\oplus\Big((V_{0})_{\lambda_{0}}\otimes(V_{1})_{\lambda_{1}\circ\alpha^{-1}}\otimes...
 \otimes(V_{n-1})_{\lambda_{n-1}\circ\alpha^{1-n}}\Big).$$
 \end{itemize}
\end{prop}
\begin{proof}
Let $(V_{k},\r_{k})$ be a weight module of $\mathfrak{g}_{1}$. Then $V_{k}=\underset{\lambda\in P(V_{k})}\bigoplus(V_k)_{\lambda}$,
where $(V_{k})_{\lambda}=\{v\in V_{k}/\r_{k}(h)v_{\lambda}=\lambda(h)v, h\in \mathfrak{h}_{1}\}$. By Proposition \ref{prop4.4}, there exists $\widetilde{\r}_{k}$ and $\b_{k}\in GL(V_k)$ such that for $h\in \mathfrak{h}_{1}, v\in (V_{k})_{\lambda}$, w have 
\begin{eqnarray*}
\widetilde{\r}_{l}(\alpha^{k}(h))\b_{k}^{k}(v)=\b_{k}^{k}\r(h)v\
=\b_{k}^{k}(\lambda(h)v)
=\lambda(h)\b_{k}^{k}(v)
=\lambda\circ\alpha^{-k}(\alpha^{k}(h))
(\b_{k}^{k}(v)).
\end{eqnarray*}
So $\b_{k}^{k}((V_{k})_{\lambda})=(V_{k})_{\lambda}$ is a weight subspace of $\b^{k}_{k}(V_{k})=V_{k}$
with respect to the CSA $\alpha^{k}(\mathfrak{h}_{1})$ of $\alpha^{k}(\mathfrak{g}_{1})$
of weight $\lambda\circ\alpha^{-k}$.
Moreover we have, $\b^{k}_{k}((V_{k})_{\lambda})=(\b_{k}^{k}(V_{k}))_{\lambda\circ\alpha^{-k}}$ and
\begin{eqnarray*}
V_{k}=\b^{k}_{k}(V_k)
=\underset{\lambda\in P(V_k)}\bigoplus \b^{k}_{k}((V_k)_{\lambda})
=\underset{\lambda\in P(V_k)}\bigoplus(\b^{k}_{k}(V_{k}))_{\lambda\circ\alpha^{-k}}
=\underset{\lambda\in P(V_k)}\bigoplus (V_{k})_{\lambda\circ\alpha^{-k}}.
\end{eqnarray*}
Thus, a weight decomposition of $V$ is given by
$$V=\underset{\lambda\in P(V)}\bigoplus\Big((V_{0})_{\lambda_{0}}\otimes(V_{1})_{\lambda_{1}\circ\alpha^{-1}}\otimes...
 \otimes(V_{n-1})_{\lambda_{n-1}\circ\alpha^{1-n}}\Big).$$
\end{proof}
\begin{rem} Since $\alpha^{n}$ is an automorphism
of $\mathfrak{g}_{1}$, if $\alpha^{n}$ is an outer automorphism then
 $\alpha^{n}(\mathfrak{h}_{1})=\mathfrak{h}_{1}$ and $\beta^{n}(V_{\lambda_{k}})=
V_{\lambda_{k}}, 0\leq k\leq n-1.$

If $\alpha^{n}$ is an inner automorphism then $\alpha^{n}(\mathfrak{h}_{1})
\neq\mathfrak{h}_{1}$ and $\beta^{n}(V_{\lambda_{k}})=
V_{\lambda_{k}\circ\alpha^{-n}}$.
\end{rem}
\begin{defns}
\begin{itemize}
\item[(1)]
We call a module  of a  simple multiplicative Hom-Lie algebra a weak weight module if
it is a weight module of the induced semisimple Lie algebra.
We call a weak weight subspace of a module of a simple multiplicative Hom-Lie algebra, every weight subspace
of the weight module of the induced Lie algebra.
\item[(2)] A strong weight module $(V,\r_{\b},\b)$ of $(\mathfrak{g},[\cdot,\cdot]_{\alpha},\alpha)$ is a weight module of $(\mathfrak{g},[\cdot,\cdot])$ such that $\b$ transforms  weight subspaces to
     weight subspaces.
    \end{itemize}
    \end{defns}
Let $\mathfrak{g}$ be a simple Lie algebra. Then, every irreducible finite dimensional
 module of $\mathfrak{g}$
is a highest weight module and its highest weight is a dominant weight.
Let $P^{+}$ be the set of dominant weights of $\mathfrak{g}$ and $\lambda\in P^{+}$. Denote
such irreducible module by $V(\lambda)$ and $v_{\lambda}$ its highest weight vector.
\begin{prop}
Keeping the same hypothesis as in Proposition \ref{weightmodul} and let $(V,\r_{\b},\b)$
be a finite dimensional irreducible module of $(\mathfrak{g},[\cdot,\cdot]_{\alpha},\alpha)$. Then $(V,\r_{\b},\b)$ is a weak weight module. \\ Moreover, there exists $(\lambda_{0},\lambda_{1},...,\lambda_{n-1})$ dominant weights for $\mathfrak{g}_{1}$
such that $$V=V(\lambda_{0})\otimes V(\lambda_{1}\circ\alpha^{-1})\otimes...\otimes
 V(\lambda_{n-1}\circ\alpha^{-(n-1)})$$ is an irreducible $\mathfrak{g}$-weight module of
 highest weight $(\lambda_{0},\lambda_{1}\circ\alpha^{-1},...,\lambda_{n-1}\circ\alpha^{1-n})$.\\
Furthermore, there exists $\b_{0}\in GL(V(\lambda_{0})),\b_{1}\in GL(V(\lambda_{1})),...,\b_{n-1}
\in GL(V(\lambda_{n-1}))$ such that $\b^{k}_{k}(V(\lambda_{k}))=V(\lambda_{k}\circ\alpha^{-k}).$
\end{prop}
\begin{proof}
Let us assume that $V(\lambda_{k})$ is a highest weight module of highest weight $\lambda_{k}\in\mathfrak{h}_{1}^{*}$ of $\mathfrak{g}_{1}$ and highest weight vector $v_{\lambda_{k}}, (v_{\lambda_{k}} \neq 0)$. Since $\alpha^{k}(\mathfrak{g}_{1})
=\alpha^{k}(\mathfrak{n}_{1}^{+})\oplus\alpha^{k}(\mathfrak{h}_{1})
\oplus\alpha^{k}(\mathfrak{n}_{1}^{-})$, by Proposition \ref{prop4.4} we have for all $x\in
\mathfrak{n}^{+},~~\widetilde{\r}_{k}(\alpha^{k}(x))\circ\b^{k}_{k}(v_{\lambda_k})
=\b^{k}_{k}\circ\r_{k}(x)(v_{\lambda_{k}})=\b_{k}^{k}(\overrightarrow{0})=
\overrightarrow{0}$.\\So $\b^{k}_{k}(v_{\lambda_k})$ is a highest weight
vector of highest weight $\lambda_{k}\circ\alpha^{-k}$ of $\b^{k}_{k}(V(\lambda_k))=V(\lambda_k\circ\alpha^{-k})$
as $\alpha^{k}(\mathfrak{g}_{1})$-module and we have $(v_{\lambda_{0}}\otimes\b_{1}(v_{\lambda_1})\otimes...\otimes \b^{n-1}_{n-1}(v_{\lambda_{n-1}}))$ is a highest weight vector of $V$ considered as $\mathfrak{g}$-module.\\
Let $\b=\b_{0}\otimes...\otimes\b_{n-1}$ and set $\r_{\b}:=\b\circ\r,$ where $\r$ is the representation of the induced semisimple Lie algebra. Then, $(V,\r_{\b},\b)$ is a representation of the simple multiplicative Hom-Lie algebra $(\mathfrak{g}, [\cdot,\cdot]_{\alpha},\alpha)$. It turns out that  $\b_{k}$ do not necessary transforms  weight subspaces
to weight subspaces. If it is the case then the weight module of the induced Lie algebra
becomes weight module for the multiplicative Hom-Lie algebra
$(\mathfrak{g},[\cdot,\cdot]_{\alpha},\alpha)$. The same holds for highest
weight modules.
\end{proof}

Let $M(\lambda)$ be a Verma module of $\mathfrak{g}$ of highest weight $\lambda\in P$.
If there exists $\alpha\in \Delta_{+}$ (the set of positive roots) such that
$\lambda(\alpha)\in \mathbb{N}$,  then $M(\lambda)$ is reducible
and there exists a maximal submodule $\overline{M(\lambda)}$ of $M(\lambda)$ such that
$V(\lambda):=M(\lambda)/\overline{M(\lambda)}$ is irreducible. Otherwise,
$\lambda(\alpha)\in \mathbb{R}\backslash\mathbb{N}$ and $M(\lambda)$ is irreducible.
\begin{rem}
Let $M(\lambda_{0}),...,M(\lambda_{n-1})$ be $n$ Verma modules of $\mathfrak{g}_{1}.$
Then $M(\lambda):=M(\lambda_{0})\otimes M(\lambda_{1}\circ\alpha^{-1}) \otimes...\otimes M(\lambda_{n-1}\circ\alpha^{-(n-1)})$ is a Verma module of
$\mathfrak{g}=\mathfrak{g}_{1}\oplus\alpha(\mathfrak{g}_{1})\oplus...\oplus \alpha^{n-1}(\mathfrak{g}_{1})$ of highest weight $\lambda=(\lambda_{0},\lambda_{1}\circ\alpha^{-1},...,\lambda_{n-1}\circ\alpha^{-(n-1)})$.\\
Let $(M(\lambda),\r)$ the corresponding representation of $\mathfrak{g}$.
For $\b\in GL(M(\lambda))$, we have $(M(\lambda),\r_{\b},\b)$ is a weak weight module
of $(\mathfrak{g},[\cdot,\cdot]_{\alpha},\alpha)$ where $\r_{\b}=\b\circ\r.$
\end{rem}

\section{Applications}
The simple Lie algebra  $\mathfrak{sl}(2)$ is the
smallest simple Lie algebra which plays a distinguish  role in Lie theory. In this section, we provide examples and study representations of simple Hom-Lie algebras of $\mathfrak{sl}(2)$-type.
\subsection{Representations of the $\mathfrak{sl}(2)$-type Hom-Lie algebras}
We consider the usual Lie algebra  $\mathfrak{sl}(2)$  generated by $\{ e,f,h \}$ and defined by the brackets $[h,e]=2 e, \; [h,f]=-2f, \;
[e,f]=h$. We  call  $\mathfrak{sl}(2)$-type Hom-Lie algebras or Hom-$\mathfrak{sl}(2)$, the Hom-Lie algebras  obtained by  applying  the Yau twist to $\mathfrak{sl}(2)$. The twisted algebras are obtained along algebra morphisms which are automorphisms in the case  of $\mathfrak{sl}(2)$. They are determined in \cite{N.J},
see also \cite{O.E}. 
In this section we will study in details  representations of  $\mathfrak{sl}(2)$-type Hom-Lie algebras.
The Hom-$\mathfrak{sl}(2)$-modules will be constructed using Theorem \ref{Keythm}.
 One needs first to consider the set of all morphisms
on $\mathfrak{sl}(2)$ and then seek for twistings of $\mathfrak{sl}(2)$-modules.


\begin{lem}Every diagonal Twist of $\mathfrak{sl}(2)$ is given by
a morphism $\alpha$ defined with respect to the basis  $\{e,f,h\}$ by 
$ \alpha(e)=\lambda e, \; \; \alpha(f)=\lambda ^{-1}f, \; \; \alpha(h)=h,$ where $\lambda $ is a nonzero parameter  in $\mathbb{K}$.\\
Let $(V,\r_{\b},\b)$ be a representation of Hom-$\mathfrak{sl}(2)=(\mathfrak{sl}(2), [\cdot,\cdot]_{\alpha}=\alpha ([\cdot,\cdot]),\alpha)$ where $V$ is an $(n+1)$-dimensional
vector space with basis $\{v_0,\ldots, v_n\}$.   Then
$\b(v_{i})=\lambda ^{-i}b_0v_{i}$,  $ 0 \leq i\leq n$, and  $b_0\in \mathbb{K}$.
\end{lem}
\begin{proof}
Straightforward calculations show that  the  $\mathfrak{sl}(2)$-type Hom-Lie algebras with diagonal twist  are  given by morphisms $\alpha$ of the form 
$ \alpha(e)=\lambda e, \; \; \alpha(f)=\lambda ^{-1}f, \; \; \alpha(h)=h,$ where $\lambda $ is a parameter different from $0$ in $\mathbb{K}$. Therefore Hom-$\mathfrak{sl}(2)$ is equipped with the following
 bracket, $$[h,e]_{\alpha}=2\lambda e,\; \; [h,f]_{\alpha}=-2\lambda  ^{-1}f,\; \;
[e,f]_{\alpha}=h,\;  \text{ where } \lambda \neq 0,1.$$
Let $\rho$ be a  representation of $\mathfrak{sl}(2,\mathbb{C})$ on an
$(n+1)$-dimensional vector space $V$ generated by $\{ v_{0},...,v_{n} \}$. It  is defined as follows\\
$\left \{
\begin{array}{llll}
\rho(e)v_{i}=(n-i+1)v_{i-1},\; \forall\; i=1,...,n,\ \rho(e)v_{0}=0,\\
\rho(f)v_{i}=(i+1)v_{i+1},\; \;   \forall\; i=0,...,n-1, \  \rho(f)v_{n}=0,\\
\rho(h)v_{i}=(n-2i)v_{i},\; \;  \forall\; i=0,...,n.\\
\end{array}
\right.$\\
 
In the following, we twist  the previous representation with respect to Hom-$\mathfrak{sl}(2)$.\\  Let $\b\in End(V)$ and set $\b(v_{i})=\displaystyle\sum_{j=0}^{n}a_{ij}v_{j}$. We construct  maps $\b$  that satisfy Condition \eqref{char}.\\ 
First, we apply $\rho$ to  $e$, then in the  LHS we get
$\r(\alpha(e))\b(v_{i})=\lambda\displaystyle\sum_{j=1}^{n}(n-j+1)a_{ij}v_{j-1}$ and 
on the RHS we get 
 $\b(\r(e)v_{i})=(n-i+1)\displaystyle\sum_{j=0}^{n}a_{i-1,j}v_{j}$.\\
For $j=n$, the LHS vanishes; $(n-i+1)a_{i-1,n}=0$ and so we get
\begin{equation}\label{a1n}
    a_{0,n}=....=a_{n-1,n}=0.
\end{equation} Considering the equality, one has 
$(n-i+1)\displaystyle\sum_{j=0}^{n-1}a_{i-1,j}v_{j}=\lambda\displaystyle\sum_{j=1}^{n}(n-j+1)
a_{ij}v_{j-1}$. \\Rewriting the equality we get
\begin{align*}
(n-i+1)(a_{i-1,0}v_{0}+a_{i-1,1}v_{1}+...+a_{i-1,i-1}v_{i-1}+...+a_{i-1,n-1}v_{n-1}
)\\ =\lambda(na_{i1}v_{0}+(n-1)a_{i1}v_{1}+...+(n-(i-1))a_{ii}v_{i-1}+...+
a_{in}v_{n-1}).
\end{align*}
 More precisely,  we have
\begin{align*}(n-i+1)a_{i-1,0}v_{0}+(n-i+1)a_{i-1,1}v_{1}+...+(n-i+1)a_{i-1,i-1}v_{i-1}+...
+(n-i+1)a_{i-1,n-1}v_{n-1}
\\ =\lambda na_{i1}v_{0}+\lambda (n-1)a_{i1}v_{1}+...+\lambda (n-(i-1))a_{ii}v_{i-1}+...+
\lambda a_{in}v_{n-1}.
\end{align*}
By identification we get the following system of $n$ equations\\
$\left \{
\begin{array}{llll}
(n-i+1)a_{i-1,0}=\lambda n a_{i1}\\
(n-i+1)a_{i-1,1}=\lambda (n-1) a_{i2}\\
(n-i+1)a_{i-1,2}=\lambda(n-2)a_{i3}\\
\vdots\\
(n-i+1)a_{i-1,i-1}=\lambda(n-i+1)a_{ii}\\
(n-i+1)a_{i i}=\lambda(n-i)a_{i+1,i+1}\\
\vdots\\
(n-i+1)a_{i-1,n-1}=\lambda a_{in}.
\end{array}
\right.$\\

Setting $i=1$,  we get $a_{00}=\lambda a_{11}$ and $a_{12}=...=a_{1n}$ and from \eqref{a1n} we get
\begin{equation}\label{equa1}
a_{12}=...=a_{1n}=0
\end{equation}
For  $i=2$ we have 
$\left \{
\begin{array}{llll}
(n-1)a_{10}=\lambda n a_{21}\\
(n-1)a_{11}=\lambda(n-1)a_{22}\\
(n-1)a_{12}=\lambda(n-2)a_{23}\\
\vdots\\
(n-1)a_{1,n-1}=\lambda a_{2n}
\end{array}
\right.$

Using  \eqref{a1n} and \eqref{equa1} we obtain
\begin{equation}\label{equa2}
a_{21}=a_{23}=...=a_{2,n-1}=0
\end{equation}
It remains to verify that
$a_{20}=0$.\\
For  $i=3$ we have  $\left \{
\begin{array}{llll}
(n-2)a_{20}=\lambda n a_{31}\\
(n-2)a_{21}=\lambda(n-1)a_{32}\\
a_{22}=\lambda.a_{33}\\
\vdots\\
(n-2)a_{2,n-1}=\lambda a_{3n}
\end{array}
\right.$\\
Similarly,   \eqref{equa2} leads to $a_{34}=...=a_{3n}=0$.
One may show that  $a_{30}$ and $a_{31}$ also vanish, and so one.

The case $i=n$ leads to    $a_{n-1,0}=\lambda n a_{n1}, \; 
a_{n-1,1}=\lambda(n-1)a_{n2}, 
\cdots, 
a_{n-1,n-1}=\lambda  a_{nn}.$  To check that coefficients $a_{n1}=a_{n2}=...=a_{n,n-1}=0$, we use the identity \eqref{char} with generator $f$. We have 
$\r(\alpha(f))\b(v_{i})=\lambda ^{-1}\displaystyle\sum_{j=0}^{n-1}a_{ij}(j+1)v_{j+1}$ and 
$\b(\r(f)v_{i})=(i+1)\displaystyle \sum_{j=0}^{n}a_{i+1,j}v_{j}.$
Writing the equality, we obtain  
\begin{align*}
\lambda^{-1}a_{i0}v_{1}+2\lambda^{-1}a_{i1}v_{2}+3\lambda^{-1}a_{i2}v_{3}+...+n\lambda^{-1}a_{i,n-1}v_{n}
=(i+1)a_{i+1,0}v_{0}+...
+(i+1)a_{i+1,n}v_{n}.
\end{align*}
Then we get the following system of $n$ equations

$\left \{
\begin{array}{llll}
a_{i+1,0}=0, \; a_{i+1,1}=0\\
2\lambda^{-1}a_{i1}v_{2}=(i+1)a_{i+1,2}v_{2}\\
\vdots\\
n\lambda^{-1}a_{i,n-1}v_{n}=(i+1)a_{i+1,n-1}v_{n}
\end{array}
\right.$\\

Solving the system, we obtain a diagonal matrix where 
$$\lambda^{-1} a_{ii}=a_{i+1,i+1}, \forall 1\leq i\leq n.$$

Setting $a_{00}=b_0$, we have $a_{ii}=\lambda^{-i}b_{0}$ for $i=1,\ldots,n$.
\end{proof}
We proved that $\beta$ is given by a diagonal matrix that  for all $v_{i}\in V$, $\b(v_{i})=\lambda^{-i}b_{0}v_{i}$, where $b_{0}$ is a nonzero scalar.
In the following, we characterize the action with respect to $\b$, see also  \cite{XL}.
\begin{thm}\label{keyclassification}
The representations $(V,\r_{\b},\b)$ of $\mathfrak{sl}(2)$-type Hom-Lie algebras on a
vector space $V$ equipped with a basis $\{ v_{0},...,v_{n} \} $ are given  by $\b$ such that $\b(v_{i})=\lambda ^{-i}b_{0}v_{i}, b_{0}\neq 0$ and   the  
actions defined by
\begin{eqnarray*}
 &&\r_{\b}(e)v_{i}=(n-i+1)\lambda^{-i+1}b_{0}v_{i-1},\;i=1,...,n.\\
 && \r_{\b}(f)v_{i}=(i+1)\lambda^{-i-1}b_{0}v_{i+1},\;i=0,...,n-1.\\
&& \r_{\b}(h)v_{i}=(2i-n)\lambda^{-i}b_{0}v_{i},\;i=0,...,n.
\end{eqnarray*}
\end{thm}
\begin{proof} It is straightforward using Theorem \ref{Keythm}.\end{proof}
From the foregoing theorem we can classify the irreducible $\mathfrak{sl}(2)$-weight modules
as follows.
\begin{thm}
There are precisely four types of irreducible Hom-$\mathfrak{sl}(2)$-modules. The actions are described in the following. 

\begin{enumerate}
  \item The finite-dimensional irreducible modules with basis $\{v_{0},v_{1},...,v_{n}\}$ and where\\
$\begin{array}{llll}
h.v_{i}=(2i-n)\lambda^{-i}b_{0}v_{i}, \; 0\leq i\leq n,\\
e.v_{i}=\lambda^{-i-1}b_{0}v_{i+1}, \; 0\leq i\leq n, \; \; e.v_{n}=0,\\
f.v_{i}=i\lambda^{-i+1}b_{0}(n+1-i)v_{i-1}, \;  0<i\leq n, \; \;  f.v_{0}=0,\\
\text{ with } \b(v_{i})=\lambda^{-i}b_{0}v_{i}.
\end{array}$
  \item The irreducible infinite-dimensional lowest weight Hom-$\mathfrak{sl}(2)$-modules,
with a basis of $h$-eigenvectors $\{v_{0},v_{1},...\}$ and   nonnegative integer 
$\tau$, such that \\
$\begin{array}{llll}
h.v_{i}=\lambda^{-i}b_{0}(\tau+2i)v_{i},\;  i\geq0,\\
e.v_{i}=\lambda^{-i-1}b_{0}v_{i+1},\;  i\geq0,\\
f.v_{i}=-i\lambda^{-i+1}b_{0}(\tau+i-1)v_{i-1}, \; i>0, \; \;  f.v_{0}=0,\\
\text{ with } \b(v_{i})=\lambda^{-i}b_{0}v_{i}.
\end{array}$
  \item The irreducible infinite-dimensional highest weight Hom-$\mathfrak{sl}(2)$-modules,
with a basis of $h$-eigenvectors $\{v_{0}, v_{1},...\}$ and
 $\tau\in \mathbb{Z}\cap]-\infty,0[$, such that \\
$\begin{array}{llll}
h.v_{i}=\lambda ^{-i}b_{0}(\tau-2i)v_{i},\;  i\geq0,\\
f.v_{i}=\lambda ^{-i-1}b_{0}v_{i+1},\;  i\geq0,\\
e.v_{i}=i\lambda ^{-i+1}b_{0}(\tau-i+1)v_{i-1}, \;  i>0, \; \;  e.v_{0}=0,\\
\text{ with } \b(v_{i})=\lambda^{-i}b_{0}v_{i}.
\end{array}$
  \item The irreducible infinite-dimensional Hom-$\mathfrak{sl}(2)$-module (the Hom-module of intermediate series)
   with a basis
  $\{...,v_{-2},v_{-1},v_{0},v_{1},v_{2},...\}$ such that\\
$\begin{array}{llll}
h.v_{i}=\lambda^{-i}b_{0}(\tau+2i)v_{i},\;  i\in \mathbb{Z}.\\
e.v_{i}=\lambda^{-i-1}b_0v_{i+1}, \text{ if }\; i\geq 0, f.v_{i}=\lambda ^{-i+1}b_{0}v_{i-1},\;  if\;i\leq0,\\
e.v_{i}=\frac{1}{4}\Big(\mu-(\tau+2i+1)^{2}+1\Big)\lambda ^{-i-1}b_{0}v_{i+1}, \;  \text{ if }\;i<0,\\
f.v_{i}=\frac{1}{4}\Big(\mu-(\tau+2i-1)^{2}+1\Big)\lambda ^{-i+1}b_{0}v_{i-1},\text{ if }\;i>0,\\
\text{ with } \b(v_{i})=\lambda^{-i}b_{0}v_{i}\quad \text{ and } \tau \neq \sqrt{\mu +1}, \; \mu,\tau\in \mathbb{Z}.
\end{array}$
\end{enumerate}
\end{thm}
\subsection{General method to compute Hom-$\mathfrak{sl}(2)$-modules}
Let $\rho_{\b}: \mathfrak{sl}(2)\rightarrow End(V)$ be the representation
of the Hom-$\mathfrak{sl}(2)$ and
 consider the action of the generators $\{e, f, h\}$ as follows:
 
\hspace{2cm}$ 
\rho_{\b}(e)v_{i}=\mu_{i}v_{i-1},\;
\rho_{\b}(f)v_{i}=\gamma_{i}v_{i+1}, \;
\rho_{\b}(h)v_{i}=\nu_{i}v_{i}, \text{ with } 
 \b(v_{i})=\eta_{i}v_{i}.$

The next step is to calculate the parameters $\nu_{i}, \gamma_{i},
\mu_{i},\;and\;\eta_{i}$ so that $\rho_{\b}$  be a 
Hom-$\mathfrak{sl}(2)$-module  on a $(n+1)$-dimensional vector space $V$. Straightforward  computation using Definition \ref{sheng} gives for each $i=0,...,n$ the following 5 equations\\
\begin{center}
$\left \{
\begin{array}{ll}
\nu_{i}\eta_{i}=\lambda \gamma_{i}\mu_{i+1}-\lambda ^{-1}\mu_{i}\gamma_{i-1}\; (1)\\
\mu_{i}\eta_{i}=\frac{1}{2\lambda }(\mu_{i}\nu_{i-1}-\lambda \nu_{i}\mu_{i}) \; \; \;\; (2)\\
\gamma_{i}\eta_{i}=-\frac{\lambda }{2}(\gamma_{i}\nu_{i+1}-\lambda ^{-1}\nu_{i}\gamma_{i}) \;
(3)
\end{array}\right.$  $
\left \{
\begin{array}{ll}
\eta_{i-1}\mu_{i}=\lambda \eta_{i}\mu_{i}\; \; \; \;(4)\\
\eta_{i+1}\gamma_{i}=\lambda ^{-1}\eta_{i}\gamma_{i}\; \; (5)
\end{array}
\right.$\\
\end{center}
Conditions $(4)$ and $(5)$ lead to  $\eta_{i}=(\lambda ^{-1})^{i}\eta_{0}, i\not=0$. Then\\
$(1)\Rightarrow
a\gamma_{i}\mu_{i+1}-\lambda ^{-1}\mu_{i}\gamma_{i-1}=\nu_{i}(\lambda ^{-1})^{i}\eta_{0}$.\\
$(2)\Rightarrow \frac{1}{2\lambda }\nu_{i-1}-\frac{1}{2}\nu_{i}=(\lambda ^{-1})^{i}\eta_{0}.$ \\
$(3)\Rightarrow -\frac{\lambda }{2}\nu_{i+1}+\frac{1}{2}\nu_{i}=(\lambda ^{-1})^{i}\eta_{0}.$\\
It follows that
$-\frac{1}{2}\nu_{i+1}+\frac{1}{2\lambda }\nu_{i}=(\lambda ^{-1})^{i+1}\eta_{0}$. Then,
$\frac{1}{\lambda }\nu_{i-1}-\nu_{i}=2(\lambda ^{-1})^{i}\eta_{0}.$
Thus, $\nu_{i}=\frac{1}{\lambda ^{i}}(\nu_{0}-2i\eta_{0}).$
Now for $\mu_{i}\;and\;\gamma_{i}$ we get
\vspace{-1em}
\begin{eqnarray*}
\gamma_{i}\mu_{i+1}&=& \lambda ^{-2}\mu_{i}\gamma_{i-1}+\nu_{i}(\lambda ^{-1})^{i}\eta_{0}\\
&=&\lambda ^{-2}\Big(a^{-2}\mu_{i-1}\gamma_{i-2}+\nu_{i-1}(a^{-1})^{i-1}\eta_{0}\Big)+\nu_{i}
(a^{-1})^{i}\eta_{0}\\
&=&\lambda ^{-2\times2}\Big(\lambda ^{-2}\mu_{i-2}\gamma_{i-3}+\nu_{i-2}(\lambda ^{-1})^{i-2}\eta_{0}\Big)+
(\lambda ^{-1})^{2}\nu_{i-1}(\lambda ^{-1})^{i-\lambda }\eta_{0}+\nu_{i}(\lambda ^{-1})^{i}\eta_{0}\\
&=&\lambda ^{-2\times 3}\mu_{i-2}\gamma_{i-3}+\lambda ^{-2\times2}\nu_{i-2}(\lambda ^{-1})^{i-2}\eta_{0}+
(\lambda ^{-2})\nu_{i-1}(\lambda ^{-1})^{i-1}\eta_{0}+\nu_{i}(\lambda ^{-1})^{i}\eta_{0}\\
&=&\lambda ^{-2\times3}\mu_{i-2}\gamma_{i-3}+\nu_{i-2}(\lambda ^{-1})^{i+2}\eta_{0}
+\nu_{i-1}(\lambda ^{-1})^{i+1}\eta_{0}+\nu_{i}(\lambda ^{-1})^{i}\eta_{0}\\
\cdots\\
&=&\lambda ^{-2i}\mu_{1}\gamma_{0}+\eta_{0}\displaystyle\sum^{i-1}_{k=0}\nu_{i-k}(\lambda ^{-1})^{i+k}
\Big(\nu_{0}-2(i-k)\eta_{0}\Big)\\
&=&\lambda ^{-2i}\mu_{1}\gamma_{0}+\eta_{0}\displaystyle\sum^{i-1}_{k=0}\lambda ^{-2i}\Big(\nu_{0}-2i\eta_{0}
+2k\eta_{0}\Big)\\
&=&\lambda ^{-2i}\mu_{1}\gamma_{0}+\lambda ^{-2i}\eta_{0}\Big(\displaystyle\sum^{i-1}_{k=0}\nu_{0}-
\displaystyle\sum^{i-1}_{k=0}2i\eta_{0}+\displaystyle\sum^{i-1}_{k=0}2k\eta_{0}\Big)\\
&=&\lambda ^{-2i}\mu_{1}\gamma_{0}+\lambda ^{-2i}\eta_{0}\Big(i\nu_{0}-2i^{2}\eta_{0}+
\not2i\frac{i-1}{\not2}\eta_{0}\Big)\\
&=&\lambda ^{-2i}\Big(\mu_{1}\gamma_{0}+i\eta_{0}\Big(\nu_{0}-(i+1)\eta_{0}\Big)\Big).
\end{eqnarray*}
Therefore, we get the following connections characterizing the parameters such that $\rho_{\b}$  is a 
Hom-$\mathfrak{sl}(2)$-module  on a $(n+1)$-dimensional vector space $V$:
\begin{equation*}
\eta_{i}=(\lambda ^{-1})^{i}\eta_{0}, i\not=0\;  \text{ and } 
\gamma_{i}\mu_{i+1}=\lambda ^{-2i}\Big(\mu_{1}\gamma_{0}+i\eta_{0}
(\nu_{0}-(i+1)\eta_{0})\Big).
\end{equation*}

Setting  $\eta_{0}=b_{0}, \mu_{1}=1, \gamma_{0}=n$ leads to  the same result obtained  in Theorem \ref{keyclassification}.

\end{document}